\documentclass[a4paper,10pt]{article}
\title{Uncountable trees and Cohen $\kappa$-reals}
\author{Giorgio Laguzzi (U. Freiburg)} 

\usepackage{fontenc}
\usepackage{amsmath}
\usepackage{amsfonts}
\usepackage{amssymb}
\usepackage{amsthm}
\usepackage{graphicx}
\usepackage[english]{babel}
\usepackage[utf8]{inputenc}

\newtheorem{definition2}{Definition}
\newtheorem{fact}[definition2]{Fact}
\newtheorem{lemma}[definition2]{Lemma}
\newtheorem{claim}[definition2]{Claim}
\newtheorem{remark2}[definition2]{Remark}
\newtheorem{theorem}[definition2]{Theorem}

\newtheorem*{Mobservation2}{Main Observation}
\newtheorem{example2}[definition2]{Example}
\newtheorem{corollary}[definition2]{Corollary}
\newtheorem{proposition}[definition2]{Proposition}

\newtheorem{question}{Question}
\newtheorem{question2}{Question}

\newenvironment{definition}{\begin{definition2} \upshape}{\end{definition2}}
\newenvironment{remark}{\begin{remark2} \upshape}{\end{remark2}}

\newcommand{\asacks}{\poset{AS}_\kappa}
\newcommand{\asilver}{\poset{AV}_\kappa}

\newcommand{\cantor}{2^\omega}

\newcommand{\cohen}{\poset{C}_\kappa}
\newcommand{\conc}{\smallfrown}

\newcommand{\club}{\mathrm{Club}}

\newcommand{\clubsacks}{\sacks^\club}

\newcommand{\dom}{\text{dom}}

\newcommand{\ideal}{\mathcal}

\newcommand{\ifif}{\Leftrightarrow}

\newcommand{\force}{\Vdash}

\newcommand{\gcantor}{2^\kappa}

\newcommand{\lav}{\poset{L}}
\newcommand{\laver}{\poset{L}_\kappa}
\newcommand{\laverr}{\poset{L}_{\omega_1}}

\newcommand{\mathias}{\poset{R}_\kappa}
\newcommand{\meager}{\ideal{M}}

\newcommand{\miller}{\poset{M}_\kappa}

\newcommand{\PPi}{\mathbf{\Pi}}
\newcommand{\poset}{\mathbb}

\newcommand{\restric}{{\upharpoonright}}

\newcommand{\sacks}{\poset{S}_\kappa}

\newcommand{\silver}{\poset{V}_\kappa}

\newcommand{\SSigma}{\mathbf{\Sigma}}

\newcommand{\splitting}{\mathsf{Split}}

\newcommand{\stem}{\textsf{stem}}
\newcommand{\successor}{\textsf{succ}}

\newcommand{\term}{\textsc{Term}}

\newcommand{\U}{\ideal{U}}

\begin{document}
\maketitle

\begin{abstract}
We investigate some versions of amoeba for tree forcings in the generalized Cantor and Baire spaces. This answers \cite[Question 3.20]{LLS15} and generalizes a line of research that in the standard case has been studied in \cite{LShV93}, \cite{Sp95} and \cite{Lag14}. Moreover, we also answer questions posed in \cite{FKK14} by Friedman, Khomskii and Kulikov, about the relationships between regularity properties at uncountable cardinals. We show $\SSigma^1_1$-counterexamples to some regularity properties related to trees without club splitting. In particular we prove a strong relationship between the Ramsey and the Baire properties, in slight contrast with the standard case. 
\end{abstract}

\section{Introduction}
This paper is concerned with forcings consisting of uncountable trees. In particular we focus on some issues about pure decision and Cohen $\kappa$-reals, attacking some points raised in \cite[Question 3.20]{LLS15}. The main part of the paper is centered around the investigation of amoeba forcings, and more generally, the effects of adding uncountable generic trees over models of ZFC. The importance of such a topic is that it has crucial applications in questions concerning cardinal invariants associated with tree-ideals and regularity properties. In the standard case, such a topic has been extensively studied; see \cite{LShV93}, \cite{Sp95}, \cite{Bre98} and \cite{Lag14}, \cite{Sp15} for important results in the context of $2^\omega$ and $\omega^\omega$.

When dealing with trees on $\omega$, even if the most natural versions of amoeba usually do not have pure decision, some refinements can be defined in order to even get the Laver property. This is indeed possible for Sacks, Miller, Laver and Mathias forcing, whereas in \cite{Sp15} Spinas has shown this cannot be done for Silver forcing. Rather surprisingly, we show that the situation with trees on $\kappa > \omega$ is completely different, and we are going to show that pure decision gets \emph{very often} lost.  
In the last section, we also present some results about regularity properties for tree-forcings at $\kappa$, showing $\SSigma^1_1$-counterexamples for Mathias and Laver measurability even without club splitting, and we obtain an interesting and rather surprising result connecting the generalizations of Ramsey property and the property of Baire. This contributes to solve some questions raised by Friedman, Khomskii and Kulikov concerning the regularity properties diagram at $\kappa$ (see \cite{FKK14}).

We remark that this reaserch field has largely spread out in the last years among set theorists. In particular, the problems analysed in this paper are part of those collected in \cite{LLS15}, which is the output of a series of workshops which took place in Amsterdam (November 2014), Hamburg (September 2015) and Bonn (September 2016).

\section{Definitons and notation}

Throughout the paper we assume $\kappa$ be an uncountable regular cardinal and $\kappa^{<\kappa}=\kappa$.
The elements in $\lambda^\kappa$ are called $\kappa$-reals or $\kappa$-sequences, where $\lambda$ is also a regular cardinal, usually $\lambda=2$ or $\lambda=\kappa$.  Given $s,t \in \lambda^\kappa$ we use the standard notiation: $s \perp t$ iff neither $s \subseteq t$ nor $t \subseteq s$ (and we say $s$ and $t$ are incompatible).
The following notation is also used.
\begin{itemize}

\item A \emph{tree} $T \subseteq \lambda^{<\kappa}$ is a subset closed under initial segments and its elements are called \emph{nodes}. We consider $< \kappa$-closed trees $T$, i.e., for every $\subseteq$-increasing sequence of length $<\kappa$ of nodes in $T$, the supremum (i.e., union) of these nodes is still in $T$. Moreover, we abuse of notation denoting by $|t|$ the order type of $\dom(t)$ (such a choice is rather standard in the literature). 
\item We say that a $<\kappa$-closed tree $T$ is \emph{perfect} iff for every $s \in T$ there exist $t \supseteq s$ and $\alpha, \beta \in \lambda$, $\alpha \neq \beta$, such that $t^\conc \alpha \in T$ and $t^\conc \beta \in T$; we call such $t$ a \emph{splitting node} (or \emph{splitnode}) and set $\splitting(T):=\{ t \in T: t \text{ is splitting} \}$.
\item 
We say that a splitnode $t \in T$ has \emph{order type} $\alpha$ (and we write $t \in \splitting_\alpha(T)$) iff 
\[ 
|\{ s \in T: s \subsetneq t \land s \in \splitting(T) \}|=\alpha.
\] 
\item $\successor(t,T):= \{ \alpha \in \lambda: t^\conc \alpha \in T \}$, for $t \in T$.
\item $t \in T$ is a \emph{terminal node} iff there are no $s \supsetneq t$ such that $s \in T$, and we write $t \in \term(T)$. 
\item For every tree $T$ we define the \emph{boundary of $T$} $b(T)$ as:
\[
b(T):=\{ t\notin T: \forall s \subsetneq t (s \in T)  \}.
\]
\item Given a tree $S$, let $cl(S)$ denotes the $<$-closure of $S$, i.e., $t \in cl(S)$ iff either $t \in S$ or $t$ is the limit of a $\subseteq$-increasing sequence of length $<\kappa$ of nodes in $S$.
\item We say that $T$ \emph{end-extends} $S$ iff $T \supseteq S$ and for every $t \in T \setminus S$ there exists $s \in \term(cl(S))$ such that $s \subseteq t$.
\item $\stem(T)$ is the longest node in $T$ which is compatible with every node in $T$; $T_t:= \{ s \in T: s \text{ is compatible with } t \}$.

\item Let $p \subset T$ be $<\kappa$-closed, we define
$T {\downarrow} p := \{ t \in T: \exists s \in \term(p) (s \subseteq t \vee t \subseteq s) \}$.
 
\item $[T]:= \{ x \in \lambda^\kappa: \forall \alpha < \kappa (x \restric \alpha \in T) \}$ is called the \emph{set of branches} (or \emph{body}) of $T$. 
\end{itemize}

\vspace{2mm}

We say that a poset $\poset{P}$ is a \emph{tree-forcing} if the conditions are perfect trees in $\lambda^{<\kappa}$ with the property that if $T \in \poset{P}$ and $t \in T$, then $T_t \in \poset{P}$ too. The ordering is the inclusion $\subseteq$. The generic $\kappa$-real added by a tree-forcing $\poset{P}$ is $x_G:= \bigcup_{T \in G} \stem(T)$.

Along the paper we are going to introduce several types of tree-forcings: Sacks, Silver, Miller, Cohen, Laver and Mathias. 
We remark that some popular forcings can be seen as tree-forcings, even if this might not be evident \emph{a priori}. For instance:
\begin{itemize}
\item $\kappa$-Cohen forcing $\cohen:= ( 2^{<\kappa}, \supseteq )$ can be seen as a tree-forcing by associating $s \in 2^{<\kappa}$ with the tree $T_s:=\{ t \in 2^{<\kappa}: t \supseteq s \}$. Also we will often write $[s]:= \{ x \in 2^\kappa: x \supset s \}$ instead of $[T_s]$.
\item $\kappa$-Mathias forcing 
\[
\mathias:=\{ (s,x): s \in 2^{<\kappa}, x \in 2^\kappa, s \subset x \land x \restric |s| = s  \},
\]
and we define $(t,y) \leq (s,x)$ if and only if 
\[
s \subseteq t \land \forall i (|s|\leq i<|t| \Rightarrow t(i) \leq  x(i)) \land \forall i<\kappa(y(i)\leq x(i)).
\] 

Again, one can associate $(s,x) \in \mathias$ with a tree $T_{(s,x)}$ as follows: 
\[
T_{(s,x)}:= \{ t \in 2^{<\kappa}: t \subseteq s \vee  \big( t \supseteq s \land \forall i < |t| (t(i) \leq  x(i)) \big)  \}.
\]
\end{itemize}
Each section is indeed devoted to focus on a particular kind of trees. The specific definitions are given at the beginning of the corresponding section.

\begin{definition}(Tree-ideals and tree-measurability) \label{def:meager} Let $\poset{P}$ be a tree-forcing and let $X$ be a set of $\kappa$-reals. We define:
\begin{itemize}
\item $X$ is \emph{$\poset{P}$-open dense} iff $\forall T \in \poset{P} \exists T' \leq  T ([T']\subseteq X)$. The complement of a $\poset{P}$-open dense set is called \emph{$\poset{P}$-nowhere dense}. $X$ is \emph{$\poset{P}$-meager} iff it can be covered by a $\leq \kappa$-size union of $\poset{P}$-nowhere dense sets. The ideal of $\poset{P}$-meager sets is denoted by $\ideal{I}_\poset{P}$. (The complement of a $\poset{P}$-meager set is called $\poset{P}$-comeager.)  
\item $X$ is $\poset{P}$-measurable iff for every $T \in \poset{P}$ there is $T' \leq T$, such that
\[
[T'] \cap X \in \ideal{I}_\poset{P} \text{ or } [T'] \setminus X \in \ideal{I}_\poset{P}.
\]
\end{itemize}
\end{definition}

\begin{definition} \label{def:amoeba}
\begin{itemize}
\item Let $\poset{P}$ be a tree-forcing. We say that $T \in \poset{P}$ is an \emph{absolute $\poset{P}$-generic tree over $V$} if for every forcing extension $N \supseteq V$ via a $<\kappa$-closed poset 
\[
N \models T \in \poset{P} \land \forall x \in [T](x \text{ is $\poset{P}$-generic over } V). 
\] 
\item We say that $\poset{AP}$ is an \emph{amoeba forcing for $\poset{P}$} whenever $\poset{AP}$ adds an absolute generic tree $T \in \poset{P}$ over $V$.
\end{itemize} 
\end{definition}

\begin{definition}($\kappa$-Axiom A)
Let $\poset{P}$ be a forcing notion. We say that $\poset{P}$ satisfies $\kappa$-Axiom A iff there is a sequence $\{\leq_\alpha: \alpha<\kappa  \}$ of partial orders satisfying the following properties:
\begin{enumerate}
\item $\leq_0 = \leq$, and for every $\alpha<\beta$, $\leq_\beta \subseteq \leq_\alpha$;
\item if $\{ p_\alpha:\alpha<\kappa \} \subseteq \poset{P}$ is such that for every $\alpha<\beta$, $p_\beta \leq_\alpha p_\alpha$, then there is $q \in \poset{P}$ such that for all $\alpha < \kappa$, $q \leq_\alpha p_\alpha$ (such $q$ is called \emph{fusion});
\item if $A \subseteq \poset{P}$ is an antichain, $p \in \poset{P}$ and $\alpha < \kappa$, then there is $q \leq_\alpha p$ such that $\{ p' \in A: p' \text{ and } q \text{ are compatible}  \}$ has size $\leq \kappa$. 
\end{enumerate}
\end{definition}

\begin{definition}
Let $\poset{P}$ be a forcing satisfying $\kappa$-Axiom A and $\{ \leq_\alpha: \alpha < \kappa\}$ be the corresponding sequence of orders. We say that $\poset{P}$ satisfies \emph{pure decision} iff for every formula $\varphi$, every $p \in \poset{P}$ there exists $q \leq_0 p$ deciding $\varphi$, i.e., either $q \force \varphi$ or $q \force \neg \varphi$.
\end{definition}

\section{$\kappa$-Sacks trees}

\begin{definition}
A tree $T \subseteq 2^{<\kappa}$ is called \emph{club $\kappa$-Sacks} ($T \in \sacks^\club$) iff $T$ is perfect and for every $x \in [T]$, $\{ \alpha<\kappa: x\restric \alpha \in \splitting(T)  \}$ is closed unbounded (briefly called ``club", from now on).
\end{definition}

This forcing was introduced by Kanamori (\cite{Kan80}) as a suitable generalization of Sacks forcing in order to obtain $<\kappa$-closure and preservation of $\kappa^+$ under $\leq \kappa$-support iteration. In more recent years, $\sacks^\club$ has been investigated by several authors, such as Friedman and Zdomskyy (\cite{FZ07}) and Friedman, Khomskii and Kulikov (\cite{FKK14}).
 
\begin{definition} \label{def:amoeba2}
$(p,T) \in \asacks^\club$ iff  the following holds:
\begin{itemize}
\item $T \in \sacks^\club$;
\item $p \subset T$ is a $<\kappa$-closed subtree of size $<\kappa$ such that if $\{ t_i: i < \delta \}$ is a $<\kappa$-sequence of splitnodes of $p$ then $\lim_{i<\delta} t_i \in \splitting(p)$ too;
\item $\term(p) \subseteq \splitting_\alpha(T)$, for some $\alpha < \kappa$. 
\end{itemize}
The order is given by:
\[
(p',T') \leq (p,T) \ifif T' \subseteq T \land p' \supseteq p. 
\]
\end{definition}
If $G$ is $\asacks^\club$-generic over $V$, put $T_G:= \bigcup \{ p: \exists T (p,T) \in G \}$.

The following result shows that $\asacks^\club$ is really an amoeba forcing for $\sacks^\club$ with respect to Definition \ref{def:amoeba}. The proof is similar to the one of \cite[Lemma 12]{Lag14}.
\begin{proposition}
Let $G$ be $\asacks^\club$-generic over $V$. Then $T_G$ is an absolute $\sacks^\club$-generic tree over $V$.
\end{proposition}

\begin{proof}
$T_G \in \clubsacks$ is clear, since any approximation $p \subseteq T_G$ has the property that limits of splitting nodes are splitting. 

We first check that \[ V[G] \models \forall x \in [T_G] (x \text{ is $\clubsacks$-generic over $V$}) \]
Fix $D \subseteq \clubsacks$ dense and put $E_D:=\{ (p,T): \forall t \in b(p) (T_t \in D) \}$. We claim $E_D$ is dense. In fact, fix $(p,T) \in \asacks^\club$, and for every $t \in b(p)$, pick $S_t \subseteq T_t$ such that $S_t \in D$. Then put $S:= \bigcup_{t \in b(p)} S_t$. Clearly $(p,S) \in E_D$. Hence, for every $D \subseteq \clubsacks$ open dense in $V$, 
$V[G] \models \forall x \in [T_G] (H_x \cap D \neq \emptyset),$
where $H_x:= \{ T \in \sacks^\club: x \in [T] \}$. Note that $H_x$ is a filter. If not, there are $T,T' \in H_x$, such that $T \cap T'$ is not in $\sacks^\club$. Let $A = T \cap T'$; then $D_A:= \{ S \in \clubsacks: [S] \cap [A] =\emptyset\}$ is dense, as for every $S \in \sacks^\club$ we can find $t \in S \notin A$. But then $H_x \cap D_A \neq \emptyset$, contradicting the fact that $x \in [T] \cap [T']=[A]$.  Note that what we have proven is 
\[
V[G] \models \varphi : \equiv \exists F \subseteq 2^{<\kappa} \forall x \in 2^\kappa( x \in [T_G] \Rightarrow \exists t \in F (t \subset x \land (T_G)_t \in D)).
\]
But $\varphi$ is a $\Sigma^1_2(\kappa^\kappa)$ statement and so for every $<\kappa$-closed forcing extension $N \supseteq V[G]$, $N \models \varphi$. Since $D \subseteq \sacks^\club$ is arbitrary, we are done.
\end{proof}

We now assume $\kappa$ be inaccessible and we consider the ordering: 
\[ 
(p',T') \leq (p,T) \ifif T' \subseteq T \land p' \text{ end-extends } p.
\]
Now we aim at showing the following. 

\begin{proposition} \label{axiomA}
Let $\kappa$ be inaccessible. $\asacks^\club$ satisfies $\kappa$-Axiom A.
\end{proposition}
The sequence of orders $\{ \leq_\alpha: \alpha<\kappa \}$ is defined as follows: 

for every $(p,T),(p',T') \in \poset{AS}^\club_\kappa$,
\[
(p',T') \leq_\alpha (p,T) \text{ iff } p'=p \land T' \leq_\alpha T,
\]
where $T' \leq_\alpha T :\ifif T' \leq T \land \splitting_\alpha(T')=\splitting_\alpha(T)$.
Proposition \ref{axiomA} follows from the following two lemmata.

\begin{lemma} \label{quasi-pure}
$\asacks^\club$ satisfies quasi pure decision, i.e., given $D \subseteq \asacks^\club$ open dense and $(p,T) \in \asacks^\club$ there is $T' \in \clubsacks$ such that $T' \leq T$, $(p,T') \in \asacks^\club$ and 
\[
\forall (q, S) \leq (p,T') ((q,S) \in D \Rightarrow (q, T' {\downarrow} q) \in D).
\] 
\end{lemma}
\begin{proof}
To simplify the notation, we give a proof for $p=\emptyset$ and leave the general case to the reader. We use the following notation: given a $<\kappa$-closed tree $q'$ of size $<\kappa$ we say that $q \subseteq q'$ is a \emph{terminal subtree} iff $\forall t \in \term(q)$, $t \in \term(q')$ too.

In the following construction, we use the following notation: 
\begin{itemize}
\item[-] $T[\beta]$ denotes the tree generated by $\splitting_{\beta}(T)$. i.e., 
\[
T[\beta]:= \{ t: \exists t' \in \splitting_\beta(T) (t \subseteq t')  \}.
\] 
\item[-]For a tree $T$, $q$ $<\kappa$-size $<\kappa$-closed subtree of $T$, and $S \leq T$ such that $S$ end-extends $q$, put
\[
T \ltimes^q S :=  \{ t \in T: (\exists t_0 \in \term(q) (t_0 \not \perp t)) \Rightarrow t \in S   \}.
\]
\end{itemize}
We build a fusion sequence $\{ T_\alpha: \alpha < \kappa \}$ by induction as follows (with $T_0=T$):
\begin{itemize}
\item[Step $\alpha+1$.] Let $\{p_\alpha^i: i < \delta_\alpha  \}$ enumerate all terminal subtrees of $T_\alpha[\alpha+1]$. Note that $\delta_\alpha \leq 2^{2^\alpha} < \kappa$, since $\kappa$ is inaccessible. 
Then we proceed by induction on $i < \delta_\alpha$ as follows:
\begin{itemize}
\item[$i=0$.] If $\exists S_\alpha^0 \leq T_\alpha$ so that $(p_\alpha^0, S^0_\alpha) \in D$, then put $T_\alpha^0:= T_\alpha \ltimes^{p_\alpha^0} S^0_\alpha$; otherwise let $T^0_\alpha:= T_\alpha$.
\item[$i+1$.]  If $\exists S_\alpha^{i+1} \leq T_\alpha^{i}$ so that $(p_\alpha^{i+1}, S^{i+1}_\alpha) \in D$, then put 
\[
T_\alpha^{i+1}:= T_\alpha^{i} \ltimes^{p_\alpha^{i+1}} S^{i+1}_\alpha;
\]
otherwise let $T^{i+1}_\alpha:= T^{i}_\alpha$.
\item[$i$ limit.] Put $T^i_\alpha:= \bigcap_{j < i} T_\alpha^j$. 
\end{itemize}
Then put $T_{\alpha+1}:= \bigcap_{i < \delta_\alpha} T_\alpha^i$. Note that $T_{\alpha+1} \leq_{\alpha} T_\alpha$.
\item[Step $\alpha$ limit.] Put $T_{\alpha}:= \bigcap_{\xi<\alpha} T_\xi$. Note that for all $\xi<\alpha$, $T_\alpha \leq_\xi T_\xi$.
\end{itemize}
Finally put $T':= \bigcap_{\alpha<\kappa} T_\alpha$. We claim that $T'$ as the required property. Indeed pick any $(q, S) \leq (\emptyset, T')$ such that $(q, S) \in D$. Choose $\alpha < \kappa$ and $i < \delta_\alpha$ such that $q= p_\alpha^i$. Then the statement $\exists S_0 (p^i_\alpha,S_0) \in D$ is satisfied (with $S_0=S$), and so $(q, T'{\downarrow} q) \leq (q, T^i_\alpha{\downarrow}q) \leq (q,S_0) \in D$.

\end{proof}

\begin{lemma}
Let $A \subseteq \asacks^\club$ be a maximal antichain, $(p,T) \in \asacks^\club$ and $\alpha < \kappa$. Then there exists $T' \in \clubsacks$, $T' \leq T$ such that $(p,T') \leq_\alpha (p,T)$ and $(p,T')$ only has $\leq \kappa$ many elements compatible in $A$. 
\end{lemma}

\begin{proof}
Let $A \subseteq \asacks^\club$ be a maximal antichain and $D_A:=\{ (p,T): \exists (q,S) \in A, (p,T) \leq (q,S) \}$ its associated open dense set. Given any $(p,T) \in \asacks^\club$, $\alpha < \kappa$, let $\{ p^i: i < \delta \}$ ($\delta < \kappa$) list all terminal subtrees of $T[\alpha+1]$ end-extending $p$ and apply Lemma \ref{quasi-pure} in order to find $T^i \subseteq T {\downarrow} p^i$, for $i < \delta$, satisfying quasi pure decision for $(p^i,T{\downarrow} p^i)$ with $D=D_A$. Then define $T'$ as the limit of the following recursive construction on $i < \delta$ (starting with $S_0=T$): successor case $i+1$: put $S_{i+1}:= S_i \ltimes^{p^i} T^i$; limit case $i$: $S_i:= \bigcap_{j < i} S_j$. Finally put
 $T' := \bigcap_{i<\delta} S_i$. We get $(p,T') \leq_\alpha (p,T)$ and 
\[
\{(q,S) \in A:  (q,S) \not \perp (p,T')  \} \subseteq \bigcup_{i < \delta} \{ (q,S) \in A: (q,S) \not \perp (p^i,T^i) \}.
\]
But for every $i < \delta$, Lemma \ref{quasi-pure} implies $\{ (q,S) \in A: (q,S) \not \perp (p^i,T^i) \}$ has size $\leq \kappa$, and so also the $\delta$-size union has size $\leq \kappa$.
\end{proof}

\begin{remark} \label{remark1}
Looking at the $\omega$-case, without particular care, amoeba forcings tend to add Cohen reals (and indeed they fail to have pure decision). For instance, the naive Sacks amoeba adds the following Cohen real $c \in \cantor$: let $T_G$ be the generic Sacks tree added by the amoeba, $z$ its leftmost branch, $\{t_n: n \in \omega\}$ the set of splitting nodes such that $z = \bigcup_{n \in \omega} t_n$. Define $c(n)$ to be 0 iff $\min \{|s|: t_n^\conc 0 \subsetneq s \land  s \in \splitting(T_G) \} \leq \min \{|s|: t_n^\conc 1 \subsetneq s \land  s \in \splitting(T_G) \}$. This construction straightforwardly generalizes to our context $\gcantor$. This kind of Cohen real can be suppressed by considering a finer version of Sacks amoeba, that not only kills this instance of Cohen real, but in the $\omega$-case actually turns out to have pure decision and the Laver property (see \cite{LShV93} and \cite{Lag14}). 
However, in the generalized framework we are considering, this kind of construction fails, since we do not have an appropriate partition property for perfect trees in $\clubsacks$. A symptom of this problem is revealed by the existence of another kind of Cohen $\kappa$-real, that seems to be more robust compared to the previous one.  To build this Cohen $\kappa$-real we fix a stationary and co-stationary subset $S \subseteq \kappa$ in the ground model. Let $\asacks^\club$ be an amoeba ordered by: 
$(p',T') \leq (p,T) \ifif T' \subseteq T \land p' \text{ end-extends } p$.
Then let $x \in [T_G]$ be the leftmost branch, where $T_G$ is the generic tree added by $\asacks^\club$, and let $\{t^x_\alpha: \alpha < \kappa  \}$ enumerate all splitting nodes that are initial segments of $x$. Then put $c(\alpha)=0$ iff $|t^x_{\alpha+1}| \in S$. It is easy to check that $c$ is a Cohen $\kappa$-real. Indeed let $(p,T) \in \asacks^\club$, $t \in p$ be the longest leftmost splitnode in $p$ and $z \in T$ be the leftmost branch. Then 
\[
C:= \{ |t'|: (t' \in \splitting(T) \land t' \supsetneq t  \land t' \subseteq z)  \}
\]  
is a club, per definition of $\asacks^\club$, and so there are both splitnodes with length in $C \cap S$ and splitnodes with length in $C \cap (\kappa\setminus S)$, as $S$ is both stationary and co-stationary. Hence we can freely select the first splitnode extending $t$ in oder to meet either $S$ or its complement, and this implies $c$ is Cohen. 
 \end{remark}

\subsection{Coding by stationary sets}
The main result in this section will be the following.
\begin{theorem} \label{sacks-amoeba-cohen}
Let $V \subseteq N$ be ZFC-models such that in $N$ there is an absolute $\clubsacks$-generic tree over $V$ and $N$ is a forcing extension of $V$ via a $<\kappa$-closed poset. Then there exists $z \in N \cap 2^\kappa$ that is Cohen over $V$. 
\end{theorem} 
We remark that this result highlights a strong difference from the standard Sacks forcing in the $\omega$-case, for which one can construct an amoeba for Sacks satisfying pure decision and the Laver property. Theorem \ref{sacks-amoeba-cohen} essentially asserts that no matter how we refine our amoeba for $\sacks^\club$, we never find a version not adding Cohen $\kappa$-reals.

To prove Theorem \ref{sacks-amoeba-cohen} we introduce a way to read off a Cohen $\kappa$-sequence from $T \in \clubsacks$ in an \emph{absolute} way. This coding will use stationary subsets of $\kappa$.

So fix $\{S_\tau: \tau \in 2^{<\kappa}\}$ family of disjoint stationary subsets of $\kappa$ in $V$. Let $\{ \lambda_\alpha: \alpha < \kappa \}$ be an increasing enumeration of all limit ordinals $< \kappa$. The set $H_\alpha$ we will refer to is meant in different ways, depending whether we are dealing with a successor or an inaccessible $\kappa$. 
\begin{itemize}
\item[] \emph{inaccessible case}: put $H_\alpha:= 2^{\leq \lambda_\alpha}$.
\item[] \emph{successor case}: let $W^0$ be a well-ordering of all $s \in 2^{\lambda_0}$ and for every $\alpha < \kappa$, $t \in 2^{\lambda_\alpha}$, let $W^{\alpha+1}_t$ be a well-ordering of all $s \in 2^{\lambda_{\alpha+1}}$ extending $t$. In what follows $ot(s)$ refers to the order type of $s \in 2^{\lambda_{\alpha+1}}$ in the well-ordering $W^{\alpha+1}_t$. Then recursively define:
\begin{itemize}
\item $H_0:= \{ t \in W^0: ot(t) < \lambda_0 \}$ and 
\item $t \in H_{\alpha+1}$ iff there exists $\{ t_\xi: \xi \leq \alpha +1\}$ with the following properties:
\begin{itemize}
\item for every $\xi \leq \alpha$, one has: $t_\xi \in 2^{\lambda_\xi}$, $t_{\xi+1} \in W^{\xi+1}_{t_\xi}$, and $ot(t_{\xi})<\lambda_{\alpha+1}$, $ot(t_{\alpha+1}) < \lambda_{\alpha+1}$,
\item $t_\lambda= \lim_{\xi < \lambda} t_\xi$, for $\lambda$ limit,  \
\item $\xi < \xi' \Rightarrow t_\xi \subseteq t_{\xi'}$,
\item $t=t_{\alpha+1}$.
\end{itemize}
\item$H_\lambda:= \bigcup_{\alpha < \lambda} H_\alpha$, for $\lambda$ limit.
\end{itemize}
Note that every $H_\alpha$ has size $< \kappa$.
\end{itemize}

\begin{lemma}[Coding Lemma] \label{coding-lemma} Let $T \in \clubsacks$, $\{ D_\xi: \xi < \kappa \}$ be a $\subseteq$-decreasing family of open dense subsets of $\cohen$ in $V$. Then there is $T' \in \clubsacks$, $T' \leq T$ such that for every $\alpha < \kappa$, there exists $\tau_\alpha \in 2^{<\kappa}$ such that 
\[
\forall t \in \splitting_{\alpha+1}(T') \forall s \in H_\alpha (|t| \in S_{\tau_\alpha} \land s^\conc \tau_\alpha \in D_\alpha).
\]
\end{lemma}

The first step is to prove the following.
\begin{claim}
Given $T \in \clubsacks$, $\alpha \in \kappa$, $\tau \in 2^{<\kappa}$, there is $T' \leq_{\alpha} T$, such that
\[
\forall t \in \splitting_{\alpha+1}(T') (|t| \in S_\tau).
\] 
\end{claim}
\begin{proof}[Proof of Claim.]
Fix $\alpha < \kappa$ and $\tau \in 2^{<\kappa}$. For every $t \in \splitting_{\alpha}(T)$, $i \in \{ 0,1 \}$, pick $\sigma(t,i) \in \splitting(T)$ such that $\sigma(t,i) \supseteq t^\conc i$ and $|\sigma(t,i)| \in S_\tau$. Note that we can do that, since $S_\tau$ is stationary and we have club many splitnodes above each $t^\conc i$. Moreover, note that for every $\tau' \neq \tau$, $|\sigma(t,i)| \notin S_{\tau'}$, since the stationary sets are pairwise disjoints. Finally let $T':= \bigcup \{ T_{\sigma(t,i)}: t \in \splitting_{\alpha}(T), i \in \{ 0,1 \}  \}$. By construction, $T' \in \clubsacks$, $T' \leq_{\alpha} T$ and has the desired property.
\end{proof}

\begin{proof}[Proof of Coding Lemma.] We build a fusion sequence $\{ T_\alpha: \alpha<\kappa \}$, with $T_{\alpha+1}\leq_{\alpha} T_\alpha$ as follows (we start with $T_0=T$):
\begin{itemize}
\item[] case $\alpha+1$: first find $\tau_\alpha \in 2^{< \kappa}$ so that $\forall s \in H_{\alpha}$, $s^\conc \tau_\alpha \in D_\alpha$; note this is possible since $|H_\alpha| < \kappa$ and $D_\alpha$ is open dense in $\cohen$. (Moreover, note that $\tau_\alpha$ is not uniquely determined, but we can simply choose the $\leq_{\text{lex}}$-least.) 

Then apply the previous claim for $T=T_\alpha$ and $\tau=\tau_\alpha$ to obtain $T_{\alpha+1} \leq_{\alpha} T_\alpha$ such that 
\[
\forall t \in \splitting_{\alpha+1}(T_{\alpha+1}) \forall s \in H_\alpha (|t| \in S_{\tau_\alpha} \land s^\conc \tau_\alpha \in D_\alpha).
\]
\item[] case $\lambda$ limit: put $T_\lambda:= \bigcap_{\alpha<\lambda} T_\alpha$.
\end{itemize}
Finally put $T':= \bigcap_{\alpha} T_\alpha$.
Note that, $\splitting_\alpha(T')=\splitting_\alpha(T_\alpha)$. Hence, by construction $T'$ has the desired properties.

\end{proof}

\begin{remark}
Given $T \in \clubsacks$, let $\{ \tau_\alpha: \alpha<\kappa \}$ be such that, for every $\alpha<\kappa$, the leftmost node $t^\alpha \in \splitting_{\alpha+1}(T)$ satisfies $|t^\alpha| \in S_{\tau_{\alpha}}$. We call $\{ \tau_\alpha: \alpha < \kappa \}$ the \emph{Cohen $\kappa$-sequence associated with $T$}. 

Note that for every $T \in \sacks^\club$ for every $\bar{\tau}:=\{ \tau_\alpha: \alpha<\kappa \} \in (2^{<\kappa})^\kappa$ there exists $T'\leq T$ such that $\bar{\tau}$ is the Cohen $\kappa$-sequence associated with $T'$. 
\end{remark}

\begin{proof}[Proof of Theorem \ref{sacks-amoeba-cohen}.]
Let $c$ be a Cohen $\kappa$-real over $N$. Let $T_G \in N$ denote the absolute $\clubsacks$-generic tree over $V$, $\bar f: 2^{<\kappa} \rightarrow \splitting(T_G)$ $\subseteq$-isomorphism, $f: 2^\kappa \rightarrow [T_G]$ the homeomorphism induced by $\bar f$, and finally let $x = f(c) \in [T_G]$. Remark that $T_G, f, \bar f \in N$, while $f(c) \in N[c] \setminus N$ obviously. Note that, in $N[c]$, $x$ is $\sacks^\club$-generic over $V$, since $T_G$ is absolute generic and $N[c] \supseteq V$ is a forcing extension via a $<\kappa$-closed poset. Let $\mathcal{A}$ be the family of all maximal antichains of $\sacks^\club$ in $V$. For every $A \in \mathcal{A}$, pick $T_A \in A$ and $s \in \cohen$ such that $s \force \dot x \in [T_A]$. For every $s \in \cohen$, one can then define, 
\[
B(s):= \bigcap \{ T_A: A \in \mathcal{A} \land s \force \dot x \in [T_A]\}.
\] 

\begin{fact}
$B(s)$ contains a tree in $\sacks^\club$.
\end{fact}
\begin{proof}[Proof of Fact.] To reach a contradiction, assume not. Define $D_{s}:= \{ t' \in 2^{<\kappa}: [\bar f(t')] \cap [B(s)]=\emptyset \}$. Note $D_s \in N$. For every $t \in 2^{<\kappa}$ there is $t^* \in \splitting(T_G)$ such that $t^* \supseteq \bar f(t)$ and $t^* \notin B(s)$; this is possible as $(T_G)_{\bar f(t)} \in \sacks^\club$ and so there is $t^* \in (T_G)_{\bar f(t)} \setminus B(s)$. Then pick $t' \in 2^{<\kappa}$ so that $\bar f(t')=t^*$ is in $D_s$ and $t' \supseteq t$. That implies $D_s$ is dense in $\cohen$. Hence, $c \cap D_s \neq \emptyset$, i.e., there is $i < \kappa$ such that $c\restric i \in D_s$, which gives $\force [\bar f (\dot c \restric i)] \cap [B(s)]=\emptyset$. But we know $s \force \dot x= f(\dot c) \in [B(s)]$, by definition. Contradiction.
\end{proof}

So we can assume for every $s \in 2^{<\kappa}$ there is $T(s) \subseteq B(s)$ in $\clubsacks$. Now let $\{ T^i: i \in \kappa \}$ enumerate all such $T(s)$'s, and let $\{  \tau^i_\alpha: \alpha < \kappa \}$ be the components of the Cohen $\kappa$-sequence associated with $T^i$. Then put $z:= \bigcup_{i < \kappa} \sigma_i$, where the $\sigma_i$'s are recursively defined as follows:
\begin{itemize}
\item $\sigma_0:= \emptyset$ 
\item $\sigma_{i+1}:= \sigma_i^\conc \tau^i_{\alpha_{i+1}}$, where $\alpha_{i+1}$ is chosen in such a way that $H_{\alpha_{i+1}} \ni \sigma_i$
\item $\sigma_i:= \bigcup_{j<i} \sigma_j$, for $i$ limit ordinal.
\end{itemize}

We aim at showing that $z$ is Cohen over $V$. Let $D \subseteq \cohen$ be open dense in $V$; we say $T$ satisfies the Coding Lemma for $D$ if the sequence is so that $D_\xi = D$, for every $\xi < \kappa$. Let $\mathbf{S}(D):=\{ T \in \clubsacks: T \text{ satisfies Coding Lemma for } D \}$ in $V$, which is a dense subset of $\clubsacks$. Let $A \subseteq \mathbf{S}(D)$ maximal antichain (note $A$ is a maximal antichain in $\sacks^\club$ as well, as $\mathbf{S}(D)$ is a dense subposet of $\sacks^\club$). Then pick $T^i \leq T_A$, for some $T_A \in A$. For every $s \in H_{\alpha_{i+1}}$ we have $s^\conc \tau^i_{\alpha_{i+1}} \in D$. By construction, $\bigcup_{j<i} \sigma_j \in H_{\alpha_{i+1}}$, and so $\sigma_{i+1} \in D$. Hence, for every $D \in V$ open dense of $\cohen$, there exists $\alpha < \kappa$ such that $z\restric \alpha \in D$, which means $z$ is Cohen over $V$.
\end{proof}

\begin{corollary}
Let $\poset{AS}$ be any amoeba for $\sacks^\club$ satisfying $<\kappa$-closure and $G$ be $\poset{AS}$-generic over $V$. Then there is $c \in 2^\kappa \cap V[G]$ which is Cohen over $V$.
\end{corollary}
\begin{proof}
It is simply a direct application of the main theorem. By definition, an amoeba for $\sacks^\club$ adds an absolute generic tree $T_G$. Then simply apply the theorem for $N=V[G]$, which satisfies the assumption, as $\poset{AS}$ is $<\kappa$-closed. 
\end{proof}

\begin{remark}
Given $T \in \clubsacks$ satisfying the Coding Lemma and \{$\tau_\alpha \in 2^{<\kappa}: \alpha <\kappa\}$ its associated $\kappa$-Cohen sequence, one can define $E^*(\alpha,T)=\bigcup_{s \in H_\alpha} [s^\conc \tau_\alpha]$ and then 
\[
E^*(T) :=  \bigcap_{\beta < \kappa} \bigcup_{\alpha \geq \beta} E^*(\alpha,T) \text{ and } E(T) := \bigcap_{t \in \splitting(T)} E^*(T_t).
\]
By construction, both $E^*(T)$ and $E(T)$ are dense $\PPi^0_2$ sets ($\kappa$-intersection of open dense). 
\end{remark}

So there is a way to associate $T \in \clubsacks$ with dense $\PPi^0_2$ sets. This might be useful to answer the following interesting and natural question.
\begin{question}
Let $\meager$ be the ideal of meager sets, $\ideal{I}_{\clubsacks}$ is the ideal of $\clubsacks$-meager sets, and $\leq_T$ denotes Tukey embedding. Is $\meager \leq_T \ideal{I}_{\clubsacks}$?
\end{question}

Note that we can prove an analogue of Proposition \ref{sacks-amoeba-cohen} by replacing $\kappa$-Cohen reals with dominating $\kappa$-reals. Recall that $z \in \kappa^\kappa$ is dominating over $V$ iff $\forall x \in \kappa^\kappa \cap V \exists \alpha < \kappa \forall \beta \geq \alpha (x(\beta) < z(\beta))$. Indeed the analogue of the Coding Lemma we need in this case is the following.

\begin{lemma} Let $T \in \clubsacks$, let $\{ t_\alpha: \alpha < \kappa \}$ denote the increasing sequence of leftmost splitting nodes in $\splitting_\alpha(T)$. Let $\{ x_\xi: \xi < \kappa \}$ be a family of $\kappa$-reals. Then there is $T' \in \clubsacks$, $T' \subseteq T$ such that for every $\alpha < \kappa$, $|t_{\alpha+1}| > \sum_{\xi \leq \alpha} x_\xi(\alpha) + |t_\alpha|$. 
\end{lemma}

This result about dominating $\kappa$-reals is not so surprising, since the same occur for the standard generic Sacks tree in the $\omega$-case.

\subsection{Other versions of Sacks amoeba}
May one generalize the Sacks forcing in order to get an amoeba with pure decision and not adding Cohen $\kappa$-reals? 

In this section we actually give a partially negative answer to this issue, by showing that even finer versions of amoeba for Sacks without club splitting have problems in killing all Cohen $\kappa$-sequences. So in what follows we are going to work with any version of amoeba for $\sacks$ being $<\kappa$-closed as a forcing notion. (For instance one might consider $\kappa$ measurable and require the set of splitting levels to be in a given normal measure on $\kappa$; we get back to this example in more details in the end of this section, as it does not play a specific role in the coming construcion).

If one analyses the proof to get an amoeba for Sacks forcing in the $\omega$-case one can realize that the main step is the following (with $\kappa=\omega$). 

\vspace{3mm}
\textbf{Partition property.}
Let $\{ T_i: i < \delta \}$, with $\delta < \kappa$ and $T_i \in \sacks$. Let $C: \prod_{i < \delta} \splitting(T_i) \rightarrow \{ 0,1 \}$ be a 2-coloring. Then there exist $T'_i \leq T_i$ such that for every $i < \delta$, $C \restric \prod_{i<\delta} \splitting(T'_i)$ is constant, i.e., there is $k \in \{ 0,1 \}$ such that $\forall \langle t_i:i<\delta \rangle \in \prod_{i<\delta}\splitting(T'_i)$, $C(\langle t_i:i<\delta \rangle)=k$.

\vspace{3mm}

We are going to build a counterexample to such a partition property in our generalized context $\kappa>\omega$. (Specifically our counterexample work for $\delta=\omega$.)

\begin{definition}
Given $T \subseteq 2^{<\kappa}$ perfect tree, we say that $T$ is \emph{$\omega$-perfect} iff there is an $\subseteq$-isomorphism $h: 2^\omega \rightarrow T$, i.e., for every $s,t \in 2^\omega$ one has $s \subseteq t \ifif h(s) \subseteq h(t)$ and $s \perp t \ifif h(s) \perp h(t)$ (roughly speaking, $T$ is $\omega$-perfect if it is an isomorphic copy of $2^\omega$ inside $2^{<\kappa}$). Then we define 
\[
\Omega:=\{ T \subseteq 2^{<\kappa}: T \text{ is $\omega$-perfect} \}.
\]
\end{definition}
Given $T \in \Omega$, $x_T$ denotes the leftmost $\omega$-branch in $T$. For every $T,T' \in \Omega$ we define
\[
T \sim T' \ifif x_T=x_{T'} \land \exists t \subseteq x_T (T_t = T'_t).
\]
It is easy to check that $\sim$ is an equivalence relation.
Pick a representative for each equivalence class. We now define the following coloring $C: \Omega \rightarrow \{ 0,1 \}$. 
 For every representative $T^*$ we put $C(T^*)=0$. Given $T \in \Omega$, pick the corresponding representative $T^* \sim T$ and let $\tau=\stem(T)$ and $\tau^*=\stem(T^*)$. Note that $\tau \subseteq \tau^*$ or $\tau^* \subseteq \tau$. 
 
 For every $t,t' \in T$, with $t \subseteq t' \subseteq x_T$, let
\[
\Delta(t,t'):= 
\begin{cases}
0
& \text{iff $|\{ s \in \splitting(T): t \subseteq s \subsetneq t'  \}|$ is even,} \\
1 & \text{else}
\end{cases}
\] 
Then define 
\[
C(T):= 
\begin{cases}
\Delta(\tau,\tau^*)
& \text{if $\tau \subseteq \tau^*$} \\
\Delta(\tau^*,\tau) & \text{if $\tau^* \subset \tau$}
\end{cases}
\] 
\begin{claim}
There is no $T \in \sacks$ homogeneous for $C$ w.r.t. $\omega$-perfect trees, i.e., there is no $T \in \sacks$ and $i \in \{ 0,1 \}$ such that $\forall T' \subseteq T, T' \in \Omega$ one has $C(T')=i$.
\end{claim}
Indeed, given any $T \in \Omega$, let $\sigma=\stem(S)$ with $S=T\cap T^*$ (where $T^*$ is the representative of $T$), and pick $\tau$ be the first splitnode of $T$ extending $\sigma^\conc 0$. Then clearly $C(T)\neq C(T_{\tau})$.

As specified above, in the following result, $\asacks$ be an amoeba for $\sacks$ satisfying $<\kappa$-closure, and defined like in Definition \ref{def:amoeba2} (but in a more general framework, not necessarily with club splitting levels).

\begin{corollary}
$\asacks$ adds Cohen $\kappa$-reals.
\end{corollary}
\begin{proof}
Given $T_G$ generic tree added via $\asacks$, we can define $z \in 2^\kappa$ as follows: First let $\{ \lambda_\alpha: \alpha <\kappa \}$ be an increasing sequence of all limit ordinals $<\kappa$ (but starting with $\lambda_0=0$), then let $\{ t_{\alpha}: \alpha < \kappa\}$ be an increasing subsequence of the leftmost splitnodes in $\splitting_{\lambda_\alpha}(T_G)$ and $q_\alpha$ be the $\omega$-perfect tree generated by $\splitting_{\lambda_\alpha+\omega}(T_{t_{\alpha}})$, i.e., the tree consisting of those nodes $s$ such that there exists $s' \supseteq s$ with $s' \in \splitting_{\lambda_\alpha+\omega}(T_{t_{\alpha}})$.  Then define $z(\alpha)=C(q_{\alpha})$, for every $\alpha < \kappa$.

To show that $z$ is Cohen we argue as follows: given $(p,T) \in \asacks$, and $w \in 2^{<\kappa}$ arbitrary, let $\bar z$ be the part of $z$ and $\{ t_\alpha: \alpha \leq \delta \}$ the  leftmost splitnodes with $t_\alpha \in \splitting_{\lambda_\alpha}(T_G)$ decided by $(p,T)$. Pick $t_\delta$ and let $p_0$ be the $\omega$-perfect tree generated by $\splitting_{\lambda_\delta +\omega}(T_{t_\delta})$. By definition of $C$, we can always find $q_0 \subseteq p_0$, $q_0 \in \Omega$, such that $C(q_0)=w(0)$. Then replace $p_0$ by $q_0$ in $T$, i.e., define $T^1$ as follows: $t \in T^1$ if and only if
\begin{itemize}
\item $t \subseteq t_\delta$, or
\item $t \perp t_\delta$ and $t \in T$, or
\item $t \supseteq t_\delta$ and $\exists s \in \term(q_0) (t \text{ and } s \text{ are compatible})$.
\end{itemize}

Then proceed by induction on $1 \leq \xi < |w|$: let $p_\xi$ be the $\omega$-perfect tree generated by $\splitting_{\lambda_{\delta+\xi}+\omega} (T^\xi_{t_{\delta+\xi}})$. Then pick an $\omega$-perfect tree $q_\xi \subseteq p_\xi$ such that $C(q_\xi)= w(\xi)$ and define $T^{\xi+1}$ by replacing $p_\xi$ with $q_\xi$ as follows: $t \in T^{\xi+1}$ if and only if
\begin{itemize}
\item $t \subseteq t_{\delta+\xi}$, or
\item $t \perp t_{\delta+\xi}$ and $t \in T^\xi$, or
\item $t \supseteq t_{\delta+\xi}$ and $\exists s \in \term(q_\xi) (t \text{ and } s \text{ are compatible})$.
\end{itemize}

(For $\xi$ limit ordinals, simply put $T^\xi := \bigcap_{\eta<\xi} T^\eta$.)

Finally let $T' :=\bigcap_{\xi < |w|} T^\xi$ and $p'$ be the tree generated by $p \cup \bigcup_{\xi < |w|} q_\xi$. By construction, $(p',T') \in \asacks$, $(p',T') \leq (p,T)$ and $(p',T') \force z \supseteq \bar z^\conc w$, and this shows that $z$ is $\kappa$-Cohen.
\end{proof}

\begin{question} Can we prove an analogue of Proposition \ref{sacks-amoeba-cohen}? In other words: can we prove that if $N \supseteq V$ is a ZFC-model containing an absolute $\sacks$-generic tree over $V$, then there is $c \in 2^\kappa \cap N$ Cohen over $V$? 
\end{question}

\begin{remark}
As we mentioned at the beginning of this section, an example of $<\kappa$-closed forcing, without splitting levels, can be obtained by working with $\kappa$ measurable. 
Let $\U$ be a normal measure on $\kappa$, and define
\[ \sacks^\U:= \{ T \in \sacks: \forall x \in [T] (\{ \alpha < \kappa: x \restric \alpha \in \splitting(T) \} \in \U) \}. \]

$(p,T) \in \poset{AS}^\U$ iff $T \in \sacks^\U$ and $p$ is an initial subtree of $T$. 
\end{remark}

\section{$\kappa$-Miller and $\kappa$-Silver trees}

The situation for $\kappa$-Miller and $\kappa$-Silver trees is rather similar to that of $\kappa$-Sacks forcing. 

\begin{definition}
A tree $T \subseteq \kappa^{<\kappa}$ is \emph{club $\kappa$-Miller} ($T \in \miller^\club$) iff 
\begin{itemize}
\item for every $s \in T$ there is $t \supseteq s$, $t \in \splitting(T)$ and $\{ \alpha \in \kappa: t^\conc \alpha \in T \}$ is club;
\item for every $x \in [T]$, $\{ \alpha \in \kappa: x \restric \alpha \in \splitting(T) \}$ is club.
\end{itemize}
A tree $T \subseteq 2^{<\kappa}$ is \emph{club $\kappa$-Silver} ($T \in \silver^\club$) iff 
\begin{itemize}
\item $T$ is perfect and for every $s,t \in T$ such that $|s|=|t|$ one has $s^\conc i \ifif t^\conc i$, for $i \in \{ 0,1 \}$;
\item $\{\alpha < \kappa: \exists t \in T (t \in \splitting(T)) \land |t|=\alpha  \}$ is club.
\end{itemize}

\end{definition}
When $\kappa$ is measurable, we can similarly define $\miller^\U$ and $\silver^\U$ by replacing ``being club" with ``being in normal measure $\U$".
Pure decision for $\kappa$-Miller forcing has been studied in detail by Brendle and Montoya (private communication: they indeed proved that $\miller^\club$ does not have pure decision and adds Cohen $\kappa$-reals, while $\miller^\U$ satisfies the $\kappa$-Laver property, and so it does not add Cohen $\kappa$-reals).
\begin{remark}
The situation occurring for $\kappa$-Silver trees is essentially the same as for $\kappa$-Sacks trees, when $\kappa$ is inaccessible; the only innocuous difference when defining the corresponding amoeba $\asilver$ is that one has to maintain the uniformity of the frozen part as well, and the same care has to be taken when doing the various fusion arguments. Apart from that, the reader can easily realize that all of the definitions and proofs in the previous section work for the $\kappa$-Silver case as well.  
\end{remark}

So we only focus on the $\kappa$-Miller case, which requires some slight modifications, though it is rather similar.
When defining the corresponding amoeba we have to require the frozen part to have size $<\kappa$. This will be crucial to have quasi pure decision, and therefore not to collapse $\kappa^+$. In what follows, $\kappa$ is inaccessible. 
\begin{definition}
We say that $(p,T) \in \poset{AM}_\kappa^\club$ iff the following hold:
\begin{itemize}
\item $T \in \miller^\club$, $p \subseteq T$ and $|p| < \kappa$, $p$ $<\kappa$-closed;
\item if $\gamma$ is limit and $\{ t_\alpha: \alpha  < \gamma \}$ is a $\subseteq$-increasing sequence of splitting nodes in $p$, then $\bigcup_{\alpha<\gamma} t_\alpha \in \splitting(p)$;
\item if $\gamma$ is limit and $\{\alpha_j: j<\gamma  \}$ is a set of ordinals in $\kappa$ such that $t^\conc {\alpha_j} \in p$, then $t^\conc \alpha^* \in p$, where $\alpha^*:= \bigcup_{j \in \gamma}\alpha_j$.
\end{itemize}
\end{definition} 

\begin{proposition}
Let $G$ be $\poset{AM}_\kappa^\club$-generic over $V$ and $T_G:=\bigcup\{ p: \exists T (p,T)\in G \}$. Then, for every $<\kappa$-closed forcing extension $N \supseteq V[G]$,
\[
N \models T_G \in \miller^\club \land \forall x \in [T_G] (x \text{ is $\miller^\club$-generic over $V$}).
\] 
\end{proposition}

\begin{proof}
Similar to the one of amoeba for Sacks. For checking that the set $H_x$ be a filter, for $T,T' \in H_x$, we argue by contradiction as follows: if $T \cap T' \notin \miller^\club$ then $D:=  \{ S : [S] \cap [T \cap T']=\emptyset \}$ is dense. Hence we should have $H_x \cap D \neq \emptyset$, i.e., $x \in [S]$ for some $S \in D$, but also $x \in [T \cap T']$.

\end{proof}

\begin{lemma} $\poset{AM}_\kappa^\club$ has quasi pure decision,  i.e., given $D \subseteq \poset{AM}_\kappa^\club$ and $(p,T) \in \poset{AM}_\kappa^\club$ there is $T' \in \miller^\club$ such that $T' \subseteq T$, $(p,T') \in \poset{AM}_\kappa^\club$ and 
\[
\forall (q, S) \leq (p,T') ((q,S) \in D \Rightarrow (q, T' {\downarrow} q) \in D).
\] 
\end{lemma}

Given $T \in \miller^\club$ let $\{ t^T_\sigma: \sigma \in \kappa^{<\kappa} \}$ be the natural enumeration of all splitting nodes of $T$ with the property that for every $\sigma, \sigma' \in \kappa^{<\kappa}$, $\sigma \leq_{\text{lex}} \sigma'$ iff $t_{\sigma} \leq_{\text{lex}} t_{\sigma'}$.
Given $T,T' \in \miller^\club$ we define $T' \subseteq_\alpha T$ iff $T' \subseteq T$ and for every $\sigma \in \alpha^\alpha$, $t^T_\sigma=t^{T'}_\sigma$.
\begin{proof}
The proof is analogous to the proof of quasi pure decision for $\poset{AS}_\kappa^\club$. The only difference is that at step $\alpha+1$, instead of considering all terminal subtrees of $T_\alpha[\alpha+1]$, we consider all terminal subtrees of the tree generated by the set $\{ t_\sigma^{T_\alpha}: \sigma \in (\alpha+1)^{(\alpha+1)}  \}$.
\end{proof}

As for $\asacks^\club$ we then get the following corollary.
\begin{corollary}
$\poset{AM}_\kappa^\club$ satisfies $\kappa$-Axiom A. 
\end{corollary}

We remark that the analogous results hold for $\poset{AM}_\kappa^\U$ as well.

$\poset{AM}_\kappa^\club$ adds Cohen $\kappa$-reals, since even $\miller^\club$ itself adds Cohen $\kappa$-reals. On the contrary $\miller^\U$ does not add Cohen $\kappa$-reals, but the reader can easily realize that we can consider a construction as in the case of $\poset{AS}_\kappa$ in order to show that $\poset{AM}_\kappa$ adds Cohen $\kappa$-reals, for any $<\kappa$-closed version of amoeba forcings.  

\section{$\kappa$-Mathias and $\kappa$-Laver trees} \label{section:laver-mathias}

\paragraph{$\kappa$-Mathias forcing.}
The $\kappa$-Mathias forcing $\mathias$ for $\kappa$ uncountable is defined as the poset of pairs $(s,A)$, where $s \subset \kappa$ of size $<\kappa$ and $A \subseteq \kappa$ of size $\kappa$ such that $\sup(s) <  \min(A)$, ordered by $(t,B) \leq (s,A) \ifif t \supseteq s \land t \restric \sup s =  s \land B \subseteq A \land t \setminus s \subseteq A$. Note that this definition is equivalent to the forcing notion given in the introduction and its analogous tree-version. As for the other tree-forcings, it might be convenient to assume some further assumptions, in order to obtain a $<\kappa$-closed forcing, having some kind of fusion (such as $A \in \club$, or $A \in \mathcal{U}$, for some normal measure $\mathcal{U}$). We remark that $\kappa$-Mathias satisfies quasi pure decision and so $\kappa$-Axiom A as well, for $\kappa$ inaccessible. The proof works exactly as in the $\omega$-case, so we can omit it. We just remark that the use of $\kappa$ inaccessible is important in the proof; in fact we have to recursively run through all $<\kappa$-size subsets of a  given set of splitting levels of order type $\alpha$, for every $\alpha < \kappa$, and we need this procedure to end in $<\kappa$-many steps, for each $\alpha$.  
The situation for $\kappa$ successor in not known and it is listed as an open question in the last section.

\begin{remark} \label{club-mathias}
Let $S \subseteq \kappa$ be stationary and co-stationary. Note that for $\mathias^\club$ we have the following three straightforward facts: 
\begin{enumerate}
\item $\mathias^\club$ adds Cohen $\kappa$-reals. Let $z$ be the canonical $\mathias^\club$-generic subset of $\kappa$, and let $z:=\{ \alpha_i: i < \kappa \}$ enumerate all its element. Then define $c \in \gcantor$ by: $c(i)=0$ iff $\alpha_{i+1} \in S$. One can easily check that $c$ is $\kappa$-Cohen, by arguing that $S$ is both stationary and co-stationary, as we did above for $\asacks^\club$.
\item $\mathias^\club$ does not have pure decision. In fact, let $(s,A) \in \mathias^\club$ and put $i = |s|$. Consider the formula 
$\varphi= \alpha_{i+1} \cap S$,
where $\{ \alpha_i:i < \kappa \}$ enumerate all elements in the $\mathias$-generic $z$.
Then $\varphi$ cannot be purely decided by $(s,A)$, by a similar argument as in point 1.
\item Let $\bar f: [\kappa]^{<\kappa} \rightarrow [\kappa]^{<\kappa}$ be a map defined as follows: for every $t \in [\kappa]^{<\kappa}$, $t:= \{\alpha_i : i \leq j \}$, put $f(t)(i)= 0 \ifif \alpha_i \in S$. Then let $f: [\kappa]^\kappa \rightarrow [\kappa]^\kappa$ be the extension induced by $\bar f$. Then $f$ is obviously continuous. Moreover, if $X \subseteq [\kappa]^\kappa$ is closed nowhere dense, then $f^{-1}[X]$ is $\mathias^\club$-meager; indeed, for every $(s,A) \in \mathias^\club$, let $\sigma= \bar f(s)$, and pick $\sigma' \supseteq \sigma$ such that $[\sigma'] \cap X =\emptyset$. Note that we can pick $s \subseteq s' \subset A$, so that $\bar f(s')=\sigma'$. Then, we get $(s',A) \leq (s,A)$ and $[s',A] \cap f^{-1}[X]=\emptyset$ (where $[s,A]:= \{ x \in [\kappa]^\kappa: x \supset s \land x \subseteq A  \}$).
\end{enumerate}  
\end{remark}

Now we want to show that \emph{any} $<\kappa$-closed version of $\kappa$-Mathias forcing adds Cohen $\kappa$-reals, and in particular has no pure decision. I thank Heike Mildenberger for suggesting me an idea about $\omega$-tuples giving me a hint for the coming construction.  \

We work with the standard version of $\kappa$-Mathias forcing, but clearly an analogous construction works for the tree-version as well. For every $a,b \in [\kappa]^\omega$, we define the following equivalence relation: $a \approx b \ifif |a \bigtriangleup b| < \omega$. We also choose a representative for any equivalence class. We then define a coloring $C: [\kappa]^\omega \rightarrow \{ 0,1 \}$ as follows:
\begin{enumerate}
\item[] for $b \in [\kappa]^\omega$, let $a$ be the representative of $[b]_\approx$. Then put:
\[
C(b):= 
\begin{cases}
0
& \text{iff $a \bigtriangleup b$ is even} \\
1 & \text{else}.
\end{cases}
\] 
\end{enumerate}
Let $x \subseteq \kappa$ be the Mathias generic and $\{ \alpha_j: j < \kappa \}$ enumerate all limit ordinals $<\kappa$. $i^x(\xi)$ denotes the $\xi$th elements of $x$. Define, for $j <\kappa$,
\[
z(j):= 
\begin{cases}
0
& \text{iff $C(\{ i^x(\xi) \in x: \alpha_j \leq  \xi < \alpha_{j+1}) \})$=0} \\
1 & \text{else}
\end{cases}
\] 
(Note that $\alpha_{j+1}= \alpha_j + \omega$ and so $C$ is well defined, since the set $\{ i^x(\xi) \in x: \alpha_j \leq  \xi < \alpha_{j+1}) \} \in [\kappa]^\omega$.)
We claim $z$ is Cohen. Fix $(s,A) \in \mathias$ and let $z_{0}$ be the $<\kappa$-initial segment of $z$ already decided by $(s,A)$. Let $t \in 2^{<\kappa}$. We are going to find $A' \subseteq A$ and $s' \supseteq s$ such that $(s',A') \leq (s,A)$ and $(s',A') \force z_{0}^\conc t \subseteq z$. This will imply $z$ be Cohen. W.l.o.g., assume $(s,A)$ exactly decides the first $\alpha_\lambda+1$-many elements in $x$. Then let $b_j:=\{ i^A(\xi) \in A: \alpha_{\lambda+ j} \leq \xi < \alpha_{\lambda+j+1}   \}$, $a_j$ be the corresponding representative, and $\xi_j:= \min (b_j \cap a_j)$. We then recursively define $b'_j \subseteq b_j$, for $j < |t|$, as follows:
\[
b'_j:=
\begin{cases}
b_j
& \text{ if $C(b_j)=t(j)$} \\
b_j \setminus \{ \xi_j \}
& \text{ if $C(b_j) \neq t(j)$} 
\end{cases}
\]
Let $\Gamma:=\{\xi_j: b'_j \neq b_j \}$ and $A':= A \setminus \Gamma$. Moreover, let $s'= s \cup \sigma$, where $\sigma := \bigcup_{j < |t|} b'_j$. Hence, for every $j < |t|$, $(s', A') \force z(\lambda+j)=t(j)$, since $(s', A') \force z(\lambda+j)= C(b'_j)=t(j)$.

Note that this construction provides a counterexample to pure decision as well. Indeed, given $(s,A) \in \mathias$, pick $\alpha_\lambda$ so large that $b:= \{ i^x(\xi) \in x: \alpha_\lambda \leq  \xi < \alpha_{\lambda+1} \}$ is not decided by $(s,A)$, where $x$ is the Mathias generic. Then the formula $\varphi := \text{``} C(b)=0 \text{''}$ cannot be purely decided by $(s,A)$.

\vspace{3mm}

\begin{proposition} \label{prop:mathias-meas}
Let $\Gamma$ be a topologically reasonable family of subsets of $\kappa$-reals, i.e. $\Gamma$ closed under continuous pre-images and intersections with closed sets.  Then 
\[
\Gamma(\mathias) \Rightarrow \Gamma(\textsc{Baire}).
\]
\end{proposition}

\begin{proof}
Let $\{ \alpha_j: j < \kappa \}$ enumerate all limit ordinals $<\kappa$ (but starting with $\alpha_0=0$) and without loss of generality we consider trees $T \in \mathias$  for which there exist $j< \kappa $ so that 
\[
\text{$\{ \xi < \kappa: \stem(T)(\xi)=1\}$ has order type $\alpha_j$.} 
\]
(Note that such trees form a dense subposet of $\mathias$, as one can always lengthen the stem with as many 1s as we need in order to catch the subsequent limit ordinal.)

Let $H$ consist of the sequences in $2^\kappa$ which are not eventually equal 0. Define $h : H \rightarrow H$ so that, for every $x \in H$,
\[
h(x)(j):= C(\{ i^x(\xi) \in x: \alpha_j \leq  \xi < \alpha_{j+1}  \}).
\]
For every $j < \kappa$, put $H_j:= \{ t \in 2^{<\kappa}: |\{ \xi: t(\xi)=1 \}| \text{ has order type $\alpha_j$} \}$. Let $h^*: \bigcup_{j < \kappa} H_j \rightarrow \bigcup_{j < \kappa} H_j$ be the function induced by $h$, i.e., $h^*$ is such that for every $x \in H$, 
\[
h(x):= \lim_{j<\kappa} h^*(x \restric \alpha_j).
\]
Note that any subset of $2^\kappa$ differs from $X \cap H$ by a set of size $\leq \kappa$, and so it does not affect either the $\mathias$-measurability or the Baire property.

It is easy to check that $h$ is continuous and surjective. Moreover, for every $T \in \mathias$ one has $h[[T]]=[h^* (\stem(T))]$.   
As an immediate consequence, for every $X \subseteq 2^\kappa$, if $h^{-1}[X]$ is $\mathias$-open dense, then $X$ is open dense.

Fix $X \in \Gamma$ and let $Y:= h^{-1}[X]$. We want to show that $X$ has the Baire property. Note that $Y \in \Gamma$ too, and so it is $\mathias$-measurable. This provides us with two possible cases.

\underline{Case 1}: there is $T \in \mathias$ such that $[T] \cap Y$ is $\mathias$-comeager, and so there are sets $B_\alpha$, $\alpha < \kappa$, so that each $B_\alpha$ is $\mathias$-open dense in $[T]$ and $\bigcap_{\alpha<\kappa} B_\alpha \subseteq Y \cap [T]$.  We claim that $X$ is comeager in $h^*(\stem(T))$. Put $t:=h^*(\stem(T))$. We aim at building a sequence $\{ U_i: i < \kappa \}$ of open dense sets in $[t]$ such that $\bigcap_{i< \kappa} U_i \subseteq [t] \cap X$, which means $X$ is comeager in $[t]$. For $\sigma \in \kappa^{<\kappa}$ define $T _\sigma \in \mathias$ such that: 
\begin{enumerate}
\item $T_{\langle \rangle}:= T$; 
\item $\bigcup_{j < \kappa} [h^*(\stem(T_{\sigma^\conc j}))] \subseteq [h^*(\stem(T_\sigma))]$ is comeager in $[h^*(\stem(T_\sigma))]$;  
\item for every $\alpha < \kappa$, $\sigma \in \kappa^{<\kappa}$ such that $|\sigma|=\alpha$, we have $\bigcup_{j<\kappa}[T_{\sigma^\conc j}] \subseteq \bigcap_{\beta \leq \alpha} B_\beta$;
\item for every $j \in \kappa$, $|\stem(T_{\sigma^\conc j})| > |\stem(T_\sigma)|$;
\item for every $\eta \in \kappa^\kappa$ there is (unique) $z \in Y$ such that $\bigcap_{i<\kappa} [T_{\eta\restric i}] = \{ z \}$;
\item $\bigcap_{i < \kappa} U_i$ can be written as $\bigcap_{i < \kappa} \bigcup_{|\sigma|=i} h[[T_{\sigma}]]$.
\end{enumerate}
 This can be done as follows.
Fix $\alpha < \kappa$ and $\sigma \in \kappa^{<\kappa}$ such that $|\sigma|=\alpha$. Given $\tau \in 2^{<\kappa}$, by definition of $h^*$ and the same argument used for proving $\mathias$ adds Cohen $\kappa$-reals, we can pick $S(\tau) \leq T_\sigma$ such that $h^*(\stem(S(\tau)))=h^*(\stem(T_\sigma))^\conc \tau$; then, by using the fact that each $B_\alpha$ is $\mathias$-open dense, we can find $T(\tau) \leq S(\tau)$ such that $[T(\tau)] \subseteq \bigcap_{\beta \leq \alpha}B_\beta$. Then let $\{ T_{\sigma^\conc j}: j < \kappa \}$ enumerate all such $T(\tau)$'s, for $\tau \in 2^{<\kappa}$. 
 
Now let $t_\sigma:= h^*(\stem(T_\sigma))$, for all $\sigma \in \kappa^{<\kappa}$. Find $a_\sigma \subseteq \kappa$ such that:
\begin{itemize}
\item for $i,j \in a_\sigma$, $[t_{\sigma^\conc i}] \cap [t_{\sigma^\conc j}]=\emptyset$
\item $\bigcup_{j \in a_\sigma} [t_{\sigma^\conc j}]$ is open dense in $t_\sigma$.
\end{itemize}
Note this can be done by refining the choices of $T_{\sigma^\conc j}$'s. 
Then define by recursion: $A_0=\{ \langle \rangle \}$, $A_{i+1}=\bigcup_{\sigma \in A_i} \{ \sigma^\conc j: j \in a_\sigma \}$, and put $U_i:= \bigcup_{\sigma \in A_i} [t_\sigma]$.

Then clearly $U:=\bigcap_{i < \kappa} U_i$ is dense in $[t]:= [h^*(\stem(T))]$. Finally, $U \subseteq X$; indeed given $y \in U$, the construction of the $t_\sigma$'s provides us with a unique $\eta \in \kappa^\kappa$ such that $y \in \bigcap_{i<\kappa} [t_{\eta\restric i}]$. Also the construction of the $T_\sigma$'s gives a unique $z \in \bigcap_{i<\kappa} [T_{\eta\restric i}]$, and $h(z)=y$. But $z \in Y:= h^{-1}[X]$, and so $y \in X$.

\underline{Case 2}: for densely many $T \in \mathias$, it holds $[T] \cap Y \in \ideal{I}_{\mathias}$. Hence, for densely many $s \in 2^{<\kappa}$ one has $[s] \cap X$ is $\kappa$-meager, which means that $X$ has the Baire property (following the notation of Definition \ref{def:meager}, the Baire property is equivalent to $\cohen$-measurablity). 
\end{proof}

By picking $h^{-1}[\club]$ we then obtain the following straightforward consequence.

\begin{corollary} \label{cor:mathias-meas}
There is a $\SSigma^1_1$ set that is not $\mathias$-measurable.
\end{corollary}

This is a rather surprising result; indeed, to our knowledge, it is the first example where a tree-measurability fails at $\SSigma^1_1$ for trees  without ``fat" splitting (e.g., club).

\paragraph{$\kappa$-Laver forcing.} First we consider $\laver^\club$, which consists of trees $T \subseteq \kappa^{<\kappa}$ such that:
\begin{itemize}
\item $\forall t \supseteq \stem(T) (t \in \splitting(T))$;
\item $\forall t \in \splitting(T) (\successor(t) \text{ is club})$;
\end{itemize}  

For $\laver^\club$ we have an analogue of Remark \ref{club-mathias}.
Like for the $\kappa$-Mathias forcing, we can consider version without club splitting. For $\kappa$ inaccessible, a standard proof shows that $\kappa$-Laver satisfies quasi pure decision and $\kappa$-Axiom A.
We want to show, for $\kappa=\omega_1$, we can build a Cohen $\omega_1$-real, and implicitly a sentence that cannot be purely decided. So let $\laverr$ denote any version of Laver forcing at $\omega_1$ with possibly any extra requirement on the splitting nodes in order to have $<\omega_1$-closure and fusion. 

Aiming at that, we first consider the following version of Laver forcing $\lav(\omega_1,\omega)$ in $\omega_1^\omega$. We say $T \in \lav(\omega_1, \omega)$ iff $T \subseteq \omega_1^{<\omega}$ is a tree such that for every $t \supseteq \stem(T)$, $|\successor(t,T)|=\omega$. Note that a diagonalization against trees in $\laverr(\omega_1,\omega)$ provides us with a Bernstein-type set $X$, i.e., $X \subseteq \omega_1^\omega$ such that for every $T \in \lav(\omega_1,\omega)$ one has $X \cap [T] \neq \emptyset$ and $X \setminus [T] \neq \emptyset$. So we can build the following Cohen $\omega_1$-real. 

Let $z \in \omega_1^{\omega_1}$ and $\{ \alpha_j:j<\omega_1 \}$ enumerate all limit ordinals $<\omega_1$ (but starting with $\alpha_0=0$) and let $A_z(j):= \langle z(\xi): \alpha_j \leq \xi < \alpha_{j+1}  \rangle$. Note we can view $A_z(j)$ as an element of $\omega_1^\omega$. Define $h_z \in 2^{\omega_1}$ as: $h_z(j)=1 \ifif A_z(j) \in X$.

Now let $x \in \omega_1^{\omega_1}$ be $\laverr$-generic and put $c=h_x$.
We claim $c$ is Cohen. Indeed, given $T \in \laverr$, let $c_T$ be the initial segment of $c$ already decided by $T$, and fix $t \in 2^{<\omega_1}$ arbitrarily. W.l.o.g. we can assume $|\stem(T)|$ be a limit ordinal. We have to find $T' \leq T$ such that $T' \force c_T^\conc t \subseteq c$. We recursively build the set $\{ \sigma_j:  j < |t| \}$ consisting of elements of $\omega_1^\omega$ in order to obtain:
\begin{itemize}
\item[$j=0$.] $\stem(T)^\conc \sigma_0 \in T$ such that $\sigma_0 \in X$ iff $t(0)=1$;
\item[$j$ successor.] $\stem(T)^\conc (\oplus_{i<j}\sigma_i)^\conc \sigma_{j} \in T$ such that $\sigma_{j} \in X$ iff $t(j)=1$ (where $\oplus_{i<j} \sigma_i$ simply consists of the concatenation of the $\sigma_i$'s, for $i<j$);
\item[$j$ limit.]  $\sigma_j:= \oplus_{i<j}\sigma_i$.
\end{itemize}
Finally put $\sigma:= \stem(T)^\conc (\oplus_{j < |t|} \sigma_j)$ and $T':= T_\sigma$. By construction, for every $j < |t|$,  $T' \force A_x(j) \in X \Leftrightarrow t(j)=1$, and so $T' \force c(j)=t(j)$, as desired.

Like for $\kappa$-Mathias forcing, this idea provides us with a counterexample to pure decision too. Indeed, given $T \in \laverr$, pick $j \in \omega_1$ ordinal large enough so that  $T$ does not decide $c(j)$. Then the formula $\varphi= \text{``} c(j)=1  \text{''}$ cannot be purely decided by $T$.

Hence, for $\kappa=\omega_1$, an analogue of Proposition \ref{prop:mathias-meas} and Corollary \ref{cor:mathias-meas} holds for Laver measurability as well.
\begin{proposition} \label{prop:laver-meas}
$\Gamma(\laverr)$ implies $\Gamma(\textsc{Baire})$, for $\Gamma$ topologically reasonable family. As a corollary, there is a $\SSigma^1_1$ set which is not $\laverr$-measurable.
\end{proposition}

Actually, if one looks at the proof, one can easily realize that it perfectly works for any $\kappa \leq 2^\omega$, as the argument for building a Bernstein sets works in such cases as well. On the contrary, if $\kappa > 2^\omega$ then we have too many trees compared to the possible branches we can select, and so the Bernstein-type construction of $X$ does not work anymore.
 It then remains open what about the case $\kappa >2^\omega$.

\section{Concluding remarks}

In \cite{Lag15} it was proven that if one drops the club splitting on the trees then it is possible to obtain a tree-measurability which can be forced for all projective sets (e.g., for Silver forcing) or in other cases for $\SSigma^1_1$ sets (e.g., for Miller forcing). On the other hand, in this paper we have seen that for $\kappa$-Mathias and $\omega_1$-Laver measurability this is subject to more restriction, as specified in Proposition \ref{prop:mathias-meas} and \ref{prop:laver-meas}. In the following table we summarize the currently known situation. In the column ``$\SSigma^1_1$-counterexample" we list all cases for which the existence of a non-measurable $\SSigma^1_1$ set is provable in ZFC; in the column ``Forceable" we list all cases for which $\SSigma^1_1$ or even projective measurability is forceable; the last column obviously exhibits the open questions. 
\begin{center}
\begin{tabular}{|l|l|l|r|}
\hline
\rule[-4mm]{0mm}{1mm}
\textsc{Forcing notion} & $\SSigma^1_1$-counterexample & Forceable & Unknown \\
\hline
\rule[-3mm]{0mm}{7mm}
\textsc{Sacks} & $\sacks^\club$ (\cite{FKK14}) & $\sacks$ (\cite{Lag15}) & $\sacks^\ideal{U}$ \\
\hline
\rule[-3mm]{0mm}{7mm}
\textsc{Silver} & $\silver^\club$ (\cite{Lag15}) & $\silver$ (\cite{Lag15})  & $\silver^\ideal{U}$ \\
\hline
\rule[-3mm]{0mm}{7mm}
\textsc{Miller} & $\miller^\club$ (\cite{FKK14}) &  $\miller$ (\cite{Lag15})  &  $\miller^\ideal{U}$ \\
\hline
\rule[-3mm]{0mm}{7mm}
\textsc{Laver} &  $\laver^\club$ (\cite{FKK14}), $\laverr^*$, $\laverr^\ideal{U}$ (Prop.\ref{prop:laver-meas}) & & $\laver$, $\laver^\ideal{U}$ , $\kappa>\omega_1$\\
\hline
\rule[-3mm]{0mm}{7mm}
\textsc{Mathias} & $\mathias^\club$ (\cite{FKK14}), $\mathias^*$, $\mathias^\ideal{U}$ (Cor.\ref{cor:mathias-meas}) & & \\
\hline
\rule[-3mm]{0mm}{7mm}
\textsc{Cohen} & $\cohen$ (\cite{HS01}) & & \\ 
\hline
\end{tabular}
\end{center} 
(Recall $\cohen$-measurability is simply the Baire property; the $*$ for $\laver$ and $\mathias$ simply mean that we require $<\kappa$-closure of the forcing together with fusion.)

\vspace{3mm}

About this type of questions, concerning the concistency of certain regularity properties for a given family of sets, we remark that an important tool used in the standard $\omega$-case is the amalgamation of Boolean algebras. It was originally introduced by Shelah in \cite{Sh84} for proving the concistency of the Baire property in the $\omega$-case for all projective sets without using an inaccessible cardinal. Other applications of Shelah's amalgamation were presented in \cite{JR93} and \cite{Lag14-bis}, were the authors proved some results about separating different notions of regularity properties. 

An interesting point to investigate would be to what extend we can generalize Shelah's amalgamation in our generalized context with $\kappa>\omega$. We know that a rough and trivial generalization cannot work properly, as we know that the Baire property fails for $\SSigma^1_1$-sets. Indeed if we look at Shelah's construction, we can realize that in general the amalgamation does not ensure $<\kappa$-closure; in the $\omega$-case this was not a point, as any tree-forcing is trivially $<\omega$-closed, and $\omega$ is preserved. The point is that amalgamation might collapse $\kappa$. A possible solution that we aim to further investigate could be to amalgamate in order to obtain strong homogeneity over a restricted set of $\kappa$-branches only, instead of all. (This idea was also used in \cite{Lag15} for proving that all projective sets are $\silver$-measurable, where we used strong homogeneity of Cohen $\kappa$-branches of a Silver tree.)   

\vspace{2mm}
About generic trees we recall the main questions that remain open.

\begin{question2}
Let $\meager_\kappa$ be the ideal of $\kappa$-meager sets, $I_{\clubsacks}$ is the ideal of $\clubsacks$-meager sets, and $\leq_T$ denotes Tukey embedding. Is $\meager_\kappa \leq_T I_{\clubsacks}$?
\end{question2}

\begin{question2} Can we prove an analogue of Proposition \ref{sacks-amoeba-cohen} for $\asacks$ (without club-splitting) and the other tree-forcings? In other words: can we prove that if $N \supseteq V$ is a ZFC-model containing an absolute $\sacks$-generic tree over $V$, then there is $c \in 2^\kappa \cap N$ Cohen over $V$? 
\end{question2}

Finally we remark that in all proofs about $\kappa$-Axiom A, we use that $\kappa$ be inaccessible. So the following is still open.
\begin{question2}
Can one prove $\kappa$-Axiom A for the amoebas and tree-forcings analysed in this paper, for $\kappa$ regular successor?
\end{question2}

\end{document}